\documentclass[11pt]{article}
\usepackage{algorithm, algpseudocode}
\usepackage{amsmath,amssymb,amsthm}
\usepackage{booktabs}
\usepackage{mathrsfs}
\usepackage{cite}
\usepackage[bookmarks=true,colorlinks,linkcolor=blue]{hyperref}
\usepackage{subfig}
\usepackage{color} 
\usepackage{graphicx}

 \vspace{0.6cm}

\newtheorem{lemma}{Lemma}[section]
\newtheorem{Defi}{Definition}[section]
\newtheorem{them}[lemma]{Theorem}
\newtheorem{lm}[lemma]{Lemma}

\newtheorem{rem}{Remark}
\newtheorem{Cor}{Corollary}

\DeclareMathOperator{\rank}{rank}

\DeclareMathOperator{\T}{\textup{T}}
\DeclareMathOperator{\diag}{\textup{diag}}
\allowdisplaybreaks

\begin{document}

\title{\textbf{Smith Normal Form and the Generalized Spectral Characterization of Graphs}\footnote{This work is supported by National Natural Science Foundation of China (Grant Nos.~11971376 and 12001006) and Fundamental Research Funds for the Central Universities (Grant No. 531118010622)}}
\author{
\small Lihong Qiu $^{{\rm a}}$  \quad \small Wei Wang $^{\rm b}$ \quad \small Wei Wang $^{{\rm a}}$\footnote{The corresponding author: E-mail
address: wang$\_$weiw@xjtu.edu.cn } \small \quad Hao Zhang $^{\rm c}$ \\
{\footnotesize$^{\rm a}$School of Mathematics and Statistics, Xi'an Jiaotong University, Xi'an 710049, P. R. China}\\
{\footnotesize$^{\rm b}$School of Mathematics, Physics and Finance, Anhui Polytechnic University, Wuhu 241000, P. R. China}\\
{\footnotesize$^{\rm c}$School of Mathematics, Hunan University, Changsha 410082, P. R. China}}
\date{}
\maketitle
\abstract \vspace{0.5cm} Spectral characterization of graphs is an important topic in spectral graph theory, which has received a lot of attention from researchers in recent years. It is generally very hard to show a given graph to be determined by its spectrum. Recently, Wang~\cite{wang2017JCTB} gave a simple arithmetic condition for graphs being determined by their generalized spectra. Let $G$ be a graph with adjacency matrix $A$ on $n$ vertices, and $W=[e,Ae,\ldots,A^{n-1}e]$ ($e$ is the all-one vector) be the walk-matrix of $G$. A theorem of Wang~\cite{wang2017JCTB} states that if $2^{-\lfloor n/2\rfloor}\det W$ (which is always an integer) is odd and square-free, then $G$ is determined by the generalized spectrum. In this paper, we find a new and short route which leads to a stronger version of the above theorem. The result is achieved by using the Smith Normal Form of the walk-matrix of $G$. The proposed method gives a new insight in dealing with the problem of generalized spectral characterization of graphs.

\noindent
\textbf{Key Words:} Graph spectra; Cospectral graphs; Determined by generalized spectrum; Smith Normal Form.\\
\noindent
\textbf{AMS Classification}: 05C50

\section{Introduction}

The spectrum of a graph encodes a lot of combinatorial information about the given graph and thus has long been a powerful tool in dealing with various problems in graph theory.

A long-standing unsolved question in spectral graph theory is ``Which graphs are determined by their spectra (DS for short)?". The problem originates from chemistry and is closely related to many other problems of central interest such as the graph isomorphism problem and a famous problem of Kac ``Can one hear the shape of a drum?".

We say two graphs are \emph{cospectral}, if they share the same spectrum. A graph $G$ is said to be \emph{determined by its spectrum} (DS for short)
if any graph having the same spectrum as $G$ is isomorphic to $G$. It is generally very hard and challenging to show a given graph to be DS. Despite many efforts, up to now, only very few families of graphs with special structural properties are known to be DS, and the techniques involved in proving them to be DS highly depend on the special properties of the spectra of these graphs and cannot be applied to general graphs. For the background and more known results about this problem, we refer the readers to \cite{DH,DH1}.

 In recent years, Wang and Xu~\cite{W1}, Wang~\cite{wang2017JCTB} considered the above problem from the perspective of the generalized spectrum. Two graphs are \emph{generalized cospectral} if they are cospectral with cospectral complements. A graph $G$ is said to be \emph{determined by its generalized spectrum} (DGS for short),
if any graph generalized cospectral with $G$ is isomorphic to $G$.

 For a given graph $G$ with adjacency matrix $A=A(G)$ on $n$ vertices, let $W=W(G)=:[e,Ae,\ldots,A^{n-1}e]$ ($e$ is the all-one vector) be the walk-matrix of $G$. In Wang~\cite{wang2017JCTB}, the author proved the following theorem.
 \begin{them}[Wang~\cite{wang2013EJC,wang2017JCTB}]\label{thm1} If $\frac{\det W(G)}{2^{\lfloor n/2\rfloor}}$ (which is always an integer) is odd and square-free, then $G$ is DGS.
\end{them}
It is not difficult to show (see Lemma 3.5 in~\cite{wang2017JCTB}) that $\frac{\det W(G)}{2^{\lfloor n/2\rfloor}}$ is odd and square-free if and only if the Smith Normal Form of $W(G)$ is as follows:
$${\rm diag}(\underbrace{1,1,\cdots,1}_{\lceil\frac{n}{2}\rceil},\underbrace{2,2,\cdots,2,2m}_{\lfloor\frac{n}{2}\rfloor}),$$
where $m$ is an odd and square-free integer.

Motivated by the above observation, in this paper, we shall give a stronger version of the above theorem by using Smith Normal Form of the walk-matrix of $G$, which significantly improves upon Theorem~\ref{thm1}; see Theorem~\ref{Main} in Section 2.2. The proof of our result is, however, new and much shorter than the original one, and hence gives a new insight in dealing with the problem of generalized spectral characterizations of graphs. The key new ingredient is an observation that whenever $G$ and $H$ are generalized cospectral, then under certain mild conditions, the null spaces of $W(G)$ and $W(H)$ coincide, over the finite field $\mathbb{F}_p$, where $p$ is a prime.

The rest of the paper is organized as follows. In Section 2, we shall give some preliminary results that will be needed later in the paper and then present the main result. In Section 3, we present the proof of Theorem~\ref{Main} in two cases: $p$ is odd and $p=2$.  Some examples and conclusions are given in Section 4 and 5, respectively.

\section{Preliminaries and main results}
\subsection{Preliminaries}
In this section, we shall give some preliminaries that will be used later in the paper.

Throughout this paper, let $G$ be a simple graph with adjacency matrix $A=A(G)$ on $n$ vertices. The \emph{spectrum} of $G$, denoted by ${\rm Spec}(G)$, is the multiset of the eigenvalues of $A(G)$. The \emph{generalized spectrum} of $G$ is the ordered pair $({\rm Spec}(G),{\rm Spec}(\bar{G}))$, where $\bar{G}$ is the complement of $G$. We say that $G$ and $H$ are \emph{generalized cospectral}, if they have the same generalized spectrum. A graph $G$ is \emph{determined by the generalized spectrum} (DGS for short) if any graph having the same generalized spectrum as $G$ is isomorphic to $G$.

The \emph{walk-matrix} of $G$ is defined as $$W=W(G):=[e,Ae,\ldots,A^{n-1}e],$$ where $e$ is the all-one vector. Note that the $(i,j)$-th entry of $W$ counts the number of walks starting from vertex $i$ with length $j-1$.


%
An integral matrix $U$ is \emph{unimodular}, if $\det(U) =\pm 1$. It is well-known that, for every integral matrix $M$ with full rank, there exist unimodular matrices $U$ and $V$ such that $M=USV$, where $S=\diag(d_{1},d_{2},\ldots,d_{n-1},d_{n})$ is a diagonal matrix with $d_{i}\mid d_{i+1}$ for $i=1,2,\ldots,n-1$, which is known as Smith Normal Form (abbreviated SNF) of the matrix $M$, and $d_{i}=d_i(M)$ is the $i$-th \emph{invariant factor} of $M$. The SNF of an integral matrix can be computed efficiently; see e.g.~\cite{ASchrijver}.

Recall that a \emph{rational orthogonal matrix} $Q$ is an orthogonal matrix with rational
entries; it is called \emph{regular} if the sum of every row (column) of $Q$ is one, i.e., $Qe=e$.

Wang and Xu~\cite{W1} initiated the study of the generalized spectral characterizations of graphs. The starting point of their study is the following lemma, which gives a simple characterization for a pair of graphs to be generalized cospectral.

\begin{lemma}[Wang and Xu~\cite{W1}; Johnson and Newman~\cite{ JohnsonNew}]\label{lortho}
Let $G$ and $H$ be two graphs with $\det W(G)\neq0$. Then $H$ is generalized cospectral with $G$ if and only if there exists a unique regular rational orthogonal matrix $Q$ such that $Q^{\T}A(G)Q=A(H)$.
\end{lemma}

Define $$\mathcal{Q}(G)=\{Q\in{RO_n(\mathbb{Q})}~|~Q^{\T}A(G)Q~\mbox{is a (0,1)-matrix }\},$$
 where $RO_n(\mathbb{Q})$ denotes the set of all $n$ by $n$ regular rational orthogonal matrices.

\begin{lm}[Wang and Xu~\cite{W1}]\label{good} Let $G$ be a graph with $\det W(G)\neq 0$. Then $G$ is DGS if and only if $\mathcal{Q}(G)$ contains only permutation matrices.
\end{lm}

According to the above lemma, in order to show that $G$ is DGS, it suffices to show that every $Q\in {\mathcal{Q}(G)}$ is a permutation matrix.
In order to do so, the following notion is proved to be useful.

\begin{Defi}[Wang and Xu~\cite{W1}] The level of a rational matrix $Q$, denoted by $\ell (Q)$ or $\ell$, is the smallest positive integer $k$ such that $kQ$ is an integral matrix.
\end{Defi}

It is clear that a regular rational orthogonal matrix is a permutation matrix if and only if $\ell(Q)=1$.

 The following lemma will be frequently used in the sequel.

\begin{lemma}\label{rem1}
 Let $X$ and $Y$ be two non-singular integral matrices such that $QX=Y$. Then $\ell\mid \gcd(d_n(X),d_n(Y))$,
 where $d_n(X)$ (resp. $d_n(Y)$) is $n$-th invariant factor of $X$ (resp. $Y$).

   \end{lemma}\label{level1}
   \begin{proof} Suppose that $X=USV$, where $S={\rm diag}(d_1(X),d_2(X),\ldots,d_n(X))$ is the SNF of $X$, and $U$ and $V$ are unimodular matrices. Then we have $$Q=YV^{-1}{\rm diag}(d_1^{-1}(X),d_2^{-1}(X),\ldots,d_n^{-1}(X))U^{-1},$$
  and hence $d_n(X)Q$ is an integral matrix. By the minimality of $\ell$, we get that $\ell\mid d_n(X)$. Similarly, noting that $Q^{\T}Y=X$ and $\ell(Q)=\ell(Q^{\T})$, we have $\ell\mid d_n(Y)$. So the lemma follows.
   \end{proof}

Suppose $Q\in {\mathcal{Q}(G)}$ with level $\ell$, then $Q^{\T}A(G)Q=A(H)$ for some graph $H$. It follows that $Q^{\T}A^k(G)e=A^k(H)e$ for $k=0,1,\ldots,n-1$, which gives
that $Q^{\T}W(G)=W(H)$. By Lemma~\ref{level1}, we get that $\ell\mid \gcd(d_n(W(G)),d_n(W(H)))$. Thus, the $n$-th invariant factor of $W(G)$ (resp. $W(H)$) provides useful information about the level of $Q$. How to sharpen this observation to give more information about $\ell$ will be the main focus of this paper.

In \cite{wang2013EJC} and \cite{wang2017JCTB}, the author was able to establish
the following results.

\begin{lemma}[Wang~\cite{wang2013EJC}]\label{oddwang}
Let $Q\in {\mathcal{Q}(G)}$ with level $\ell$, and $p$ be odd prime. If $p\mid \det W(G)$ and $p^2\nmid \det W(G)$, then $p\nmid \ell$.

\end{lemma}

\begin{lemma}[Wang~\cite{wang2017JCTB}]\label{evenwang}
Let $Q\in {\mathcal{Q}(G)}$ with level $\ell$. Suppose that $2^{\lfloor n/2\rfloor}\mid\det W(G)$ and $2^{\lfloor n/2\rfloor+1}\nmid\det W(G)$. Then $\ell$ is even.

\end{lemma}

\begin{rem} In~\cite{wang2013EJC}, the author proved Lemma~\ref{oddwang} using the following strategy: Let $Q\in {\mathcal{Q}(G)}$ with level $\ell$. Suppose that $p$ is an odd prime and $p\mid \ell$. Then the author was able to show that the congruence equation $W^{\T}x\equiv 0~({\rm mod}~p^2)$ has a non-trivial solution $x\not\equiv 0~({\rm mod}~p)$, which implies $p^2\mid \det W(G)$. This contradicts the assumption of the lemma. While in~\cite{wang2017JCTB}, the strategy of proving Lemma~\ref{evenwang} is too involved to describe here; we refer the reader to the original paper.

\end{rem}
 Now suppose that $G$ is a graph such that $\frac{\det W(G)}{2^{\lfloor n/2\rfloor}}$ is odd and square-free. Let $Q\in {\mathcal{Q}(G)}$ with level $\ell$. Then, combining the above two lemmas, we have $\ell=1$ and hence $Q$ is a permutation matrix and $G$ is DGS.
Thus Theorem~\ref{thm1} follows.

 Unlike the methods in~\cite{wang2013EJC} and \cite{wang2017JCTB}, in this paper, we shall give a stronger version of Theorem~\ref{thm1} using a totally different approach.

We use $\rank_{p} X$ throughout to denote the rank of an integral matrix $X$ over the finite field $\mathbb{F}_p$.  We need the following lemma.

\begin{lemma}[Qiu et al.~\cite{QiuJiWangQ}]\label{rank}
 Let $r=\rank_{p}W(G)$. Then the first $r$ columns of $W(G)$ are linearly independent over $\mathbb{F}_{p}$.
\end{lemma}


 By Lemma~\ref{rank}, we know that $A^{r}e$ can be uniquely expressed as the linear combination of $e,Ae,\ldots,A^{r-1}e$, in other words, there exist $r$ integers $a_{0},a_{1},\ldots,a_{r-1}$ such that $a_{0}e+a_{1}Ae+\cdots+a_{r-1}A^{r-1}e+A^{r}e=0$, over $\mathbb{F}_{p}$. Now we introduce a new matrix $M=M(G)$, which plays a critical role in proving our main theorem (Theorem~\ref{Main}).
\begin{Defi}\label{defM} For a graph $G$ with adjacency matrix $A$, define $M=M(G):=a_{0}I+a_{1}A+\cdots+a_{r-1}A^{r-1}+A^{r}$.
\end{Defi}

\begin{rem} \textup{We would like to mention that a similar matrix $M$ was considered in Wang~\cite{Wangan} over $\mathbb{F}_{2}$.}
\end{rem}

\begin{Defi} With the above $M$, we introduce two matrices $\bar{W}(G)$ and $\hat{W}(G)$ as follows:
 $$\bar{W}=\bar{W}(G):=[e,Ae,\ldots,A^{r-1}e,Me,AMe,\ldots,A^{n-r-1}Me],$$ and $$\hat{W}=\hat{W}(G):=[e,Ae,\ldots,A^{r-1}e,\frac{Me}p,\frac{AMe}p,\ldots,\frac{A^{n-r-1}Me}p].$$
\end{Defi}
Clearly $\hat{W}$ is an integral matrix since $Me\equiv~0~({\rm mod}~p)$. The following lemma gives a relationship between the SNF of $W$ and that of $\hat{W}$, which generalizes a result on Eulerian graphs in~\cite{QJWEulerian}.

\begin{lemma}\label{lSNF}
Let $p$ be a prime. Suppose that the ${\rm SNF}$ of $W$ is \begin{equation}\label{SNF1}
 \textup{diag}(d_1,d_2,\ldots,d_r,d_{r+1},\ldots,d_n),
  \end{equation}
  where $p\nmid d_r$ and $p\mid d_{r+1}$. Then the ${\rm SNF}$ of $\hat{W}$ is $\diag(d_1,d_2,\ldots,d_r,\frac{d_{r+1}}p,\ldots,\frac{d_n}p).$
\end{lemma}
\begin{proof}
Note that $\bar{W}$ can be obtained from $W$ by applying a series of elementary column operations. Hence $\bar{W}$ and $W$ have the same SNF. Thus, there exist two unimodular matrices $\bar{U}$ and $\bar{V}$ such that $\bar{W}=\bar{U}N\bar{V}$, where $N$ is as shown in Eq.~\eqref{SNF1}. It follows that
\begin{equation}\label{eqeven2}
\begin{split}
\bar{U}^{-1}\bar{W} &=\diag(d_1,d_2,\ldots,d_r,d_{r+1},\ldots,d_n)\bar{V}\\
   &= \left(
                           \begin{array}{cc}
                             \Delta &  \\
                              & p\Lambda \\
                           \end{array}
                         \right)\left(
                                  \begin{array}{cc}
                                    V_{1} & V_{2} \\
                                    V_{3} & V_{4} \\
                                  \end{array}
                                \right) \\
   &=\left(\begin{array}{cc}
             \Delta V_{1} & \Delta V_{2} \\
             p\Lambda V_{3} & p\Lambda  V_{4} \\
             \end{array}
              \right),
\end{split}
\end{equation}
where $\Delta=\diag(d_1,d_2,\ldots,d_r)$ is a diagonal matrix of order $r$, $\Lambda=\diag(\frac{d_{r+1}}p,\ldots,\frac{d_n}p)$ is a diagonal matrix of order $n-r$, and
 $\bar{V}=\left(
 \begin{array}{cc}
 V_{1} & V_{2} \\
 V_{3} & V_{4} \\
 \end{array}
 \right)$ is the corresponding matrix partition of $\bar{V}$.

 Then Eq.~\eqref{eqeven2} can be rewritten as
$$[\bar{U}^{-1}e,\ldots,\bar{U}^{-1}A^{r-1}e,\bar{U}^{-1}Me,\ldots,\bar{U}^{-1}A^{n-r-1}Me] =\left(\begin{array}{cc}
             \Delta V_{1} & \Delta V_{2} \\
             p\Lambda V_{3} & p\Lambda V_{4} \\
             \end{array}
              \right),$$
and hence
$$[\bar{U}^{-1}e,\ldots,\bar{U}^{-1}A^{r-1}e,\frac{\bar{U}^{-1}Me}p,\ldots,\frac{\bar{U}^{-1}A^{n-r-1}Me}p] =\left(\begin{array}{cc}
             \Delta V_{1} & \frac{\Delta V_{2}}{p} \\
             p\Lambda V_{3} & \Lambda V_{4} \\
             \end{array}
              \right),$$
i.e.,
\begin{equation}\label{eqeven3}
  \bar{U}^{-1}\hat{W}(G)=\left(\begin{array}{cc}
             \Delta V_{1} & \frac{\Delta V_{2}}{p} \\
             p\Lambda V_{3} & \Lambda V_{4} \\
             \end{array}
              \right)=\left(\begin{array}{cc}
            \Delta & O \\
            O & \Lambda
             \end{array}\right)
            \left(\begin{array}{cc}
             V_{1} & \frac{V_{2}}{p} \\
             p V_{3} & V_{4}
              \end{array}\right).
\end{equation}
Let $\bar{V}^{'}=\left(\begin{array}{cc}
             V_{1} & \frac{V_{2}}{p} \\
             p V_{3} & V_{4} \\
             \end{array}\right)$. By the first equality of Eq.~\eqref{eqeven3}, we get that $\frac{\Delta V_{2}}{p}$ is an integral matrix, and so is $\frac{V_{2}}{p}$, since $p\nmid d_i$ for $i=1,2,\ldots,r$. Thus, $\bar{V}^{'}$ is an integral matrix.
Moreover, taking determinant on both sides of Eq.~\eqref{eqeven3} gives $\det \bar{V}^{'}=\pm 1$. Therefore, $\bar{V}^{'}$ is a unimodular matrix. Then the lemma follows immediately from the second equality of Eq.~\eqref{eqeven3}.

\end{proof}

\subsection{Main Results}

The main result of this paper is the following theorem.

\begin{them}\label{Main} Let $G$ be a graph with $\det W(G)\neq0$. Let $d_n=d_n(W(G))$ be the $n$-th invariant factor of $W(G)$. Suppose that $Q\in{\mathcal{Q}_G}$ with level $\ell$. Then we have:\\
\noindent
(i)   For an odd prime $p$, if $\rank_pW(G) =n-1$, then $\ell\mid \frac{d_n}{p}$;\\
(ii)  For $p=2$, if $\rank_2 W(G)=\lceil\frac{n}2\rceil$, then $\ell\mid \frac{d_n}{2}$.
\end{them}

As an immediate consequence, we have

\begin{Cor}Set $r=\lceil\frac{n}2\rceil$. Suppose that the Smith Normal Form of $W(G)$ is as follows:
$${\rm diag}(\underbrace{1,1,\ldots,1}_{r},\underbrace{2^{l_1},2^{l_2},\ldots,2^{l_{n-r-1}},2^{l_{n-r}}p_1^{m_1}p_2^{m_2}\cdots  p_s^{m_s}}_{n-r}),$$
where $p_i$'s are distinct odd primes for $1\leq i\leq s$. Suppose that $Q\in{\mathcal{Q}_G}$ with level $\ell$. Then $\ell\mid 2^{l_{n-r}-1}p_1^{m_1-1}p_2^{m_2-1}\cdots  p_s^{m_s-1}$.
\end{Cor}

\begin{Cor}[Wang~\cite{W1,wang2017JCTB}]
If $\frac{\det W(G)}{2^{\lfloor n/2 \rfloor}}$ is odd and square-free, then $G$ is DGS.
\end{Cor}

\begin{proof} It follows from \cite{wang2017JCTB} that $\frac{\det W(G)}{2^{\lfloor n/2 \rfloor}}$ is odd and square-free if and only if the SNF of $W(G)$ is as follows:
$${\rm diag}(\underbrace{1,1,\ldots,1}_{\lceil\frac{n}{2}\rceil},\underbrace{2,2,\ldots,2,2p_1p_2\cdots p_s}_{\lfloor\frac{n}{2}\rfloor}),$$
where $p_i$'s are distinct odd primes. Let $Q\in{\mathcal{Q}_G}$ with level $\ell$. By Corollary 1, we have $\ell \mid 2^{0}p_1^{0}p_{2}^0\cdots p_s^{0}$, i.e., $\ell\mid 1$. Thus, $\ell=1$ and $Q$ is a permutation
matrix. By Lemma~\ref{good}, $G$ is DGS.

\end{proof}

\section{Proof of Theorem~\ref{Main}}

In this section, we shall present the proof of Theorem~\ref{Main}. Before doing so, we shall describe our main strategy, as follows.

\subsection{Main ideas}
Let $G$ and $H$ be two generalized cospectral graphs. For simplicity, we write $A=A(G)$ and $B=A(H)$. Let $p$ be a prime and $r=\rank_p W(G)$. Let $Q\in{\mathcal{Q}_G}$ be a regular rational orthogonal matrix such that $Q^{T}AQ=B$ and $Qe=e$, which imply $Q^{\T}A^ke=B^ke$ for $k=0,1,\ldots,n-1$, i.e.,
 \begin{gather}
 \nonumber
Q^{\T}e=e,\\\nonumber
Q^{\T}Ae=Be,\\\nonumber
\vdots\\\label{EM}
Q^{\T}A^{r-1}e=B^{r-1}e,\\ \nonumber
Q^{\T}A^{r}e=B^{r}e,\\\nonumber
\vdots\\\nonumber
Q^{\T}A^{n-1}e=B^{n-1}e.\nonumber
\end{gather}
 Suppose that we can find some integers $a_0,a_1,\ldots,a_{r-1}$ such that
 \begin{eqnarray}\label{key1}
 \left\{
 \begin{split}
 M(G)e=a_0e+a_1Ae+\cdots+a_{r-1}A^{r-1}e+A^{r}e\equiv 0~({\rm mod}~p),\\
 M(H)e=a_0e+a_1Be+\cdots+a_{r-1}B^{r-1}e+B^{r}e\equiv 0~({\rm mod}~p).
 \end{split}
 \right.
 \end{eqnarray}
 Then multiplying both sides of the first to the $r$-th equations by $a_0, a_1,\ldots, a_{r-1}$ in Eq.~\eqref{EM}, respectively, and then adding them to the $(r+1)$-th equation generates that $Q^{\T}M(G)e=M(H)e$. Furthermore, it is easy to see that $$Q^{\T}A^{i}M(G)e=B^{i}M(H)e,$$ for $i=1,2,\ldots,n-r-1$.

 Let $$\hat{W}(G):=[e,Ae,\ldots,A^{r-1}e,\frac{M(G)e}p,\frac{AM(G)e}p,\ldots,\frac{A^{n-r-1}M(G)e}p]$$ and $$\hat{W}(H):=[e,Be,\ldots,B^{r-1}e,\frac{M(H)e}p,\frac{BM(H)e}p,\ldots,\frac{B^{n-r-1}M(H)e}p].$$
 Then both $\hat{W}(G)$ and $\hat{W}(H)$ are integral matrices, and $Q^{\T}\hat{W}(G)=\hat{W}(H)$ still holds.
   Note that $d_n(\hat{W}(G))=\frac{d_n(W(G))}p$ according to Lemma~\ref{lSNF}. It follows from Lemma~\ref{rem1} that $\ell \mid \frac{d_n(W(G))}p$, and we are done!

\begin{rem}
We would like to mention that if Eq.~\eqref{key1} holds, then the null spaces of $W(G)$ and $W(H)$ coincide, or in notations, ${\rm Null}(W(G))={\rm Null} (W(H))$, over $\mathbb{F}_p$. Actually, let $\alpha_0=(a_0,a_1,\ldots,a_{r-1},1,0,\ldots,0)^{\T}$ and $$\alpha_i=(\underbrace{0,\ldots,0}_i,a_0,a_1,\ldots,a_{r-1},1,0,\ldots,0)^{\T},$$ for $i=0,1,\ldots,n-r-1$. Then it follows from the first equation in Eq.~\eqref{key1} that $W(G)\alpha_0=0$, and hence $W(G)\alpha_i=0$ for $i=1,\ldots,n-r-1$. Note that $\alpha_i$'s
are linearly independent and thus form a basis of ${\rm Null}(W(G))$. The same is true for ${\rm Null}(W(H))$.

\end{rem}

So the primary focus of our proof is to show that there indeed exist some integers $a_0,a_1,\ldots,a_{r-1}$ such that Eq.~\eqref{key1} holds. This will be done in two cases: $p$ is an odd prime and $p=2$.

\subsection{The case that $p$ is odd }
In this subsection, we shall present the proof of Theorem~\ref{Main}~(i). For the case $p$ is an odd prime, we have $r={\rm rank}_pW(G)=n-1$. We need the following lemmas.

\begin{lemma}\label{oddroot}
 Let $G$ and $H$ be two generalized cospectral graphs with walk matrices $W(G)$ and $W(H)$, respectively. Let $\phi(x)=x^n+c_{2}x^{n-2}+\cdots+c_{n-1}x+c_n$ be their common characteristic polynomial (note that $c_{1}=0$). Suppose that ${\rm rank}_p(W(G))={\rm rank}_p(W(H))=n-1$ and $\phi(x)\equiv 0~({\rm mod}~p)$ has no two distinct roots.
Then under the condition of Theorem~\ref{Main}~(i), $W(G)x\equiv 0~({\rm mod}~p)$ and $W(H)x\equiv 0~({\rm mod}~p)$ have the same set of solutions.
\end{lemma}

\begin{proof} By Cayley-Hamilton's Theorem, we have
\begin{equation}\label{ch}
  A^n+c_{2}A^{n-2}+\cdots+c_{n-1}A+c_nI=0.
\end{equation}
Right-multiplying by $e$ on both right sides of Eq.~\eqref{ch} and then reducing modulo $p$ gives that

\begin{equation}\label{CH1}
A^ne\equiv-c_{2}A^{n-2}e-\cdots-c_{n-1}Ae-c_ne~({\rm mod}~p).
\end{equation}
Suppose that $W(G)x\equiv 0~({\rm mod}~p)$ has a solution $(a_0,a_1,\ldots,a_{n-2},1)^T$, i.e.,

\begin{equation}\label{Eq1}
A^{n-1}e\equiv-a_{n-2}A^{n-2}e-\cdots-a_1Ae-a_0e~({\rm mod}~p).
\end{equation}
Left-multiplying both sides of Eq.~\eqref{Eq1} by $A$ generates that
\begin{equation}\label{Eq2}
A^{n}e\equiv-a_{n-2}A^{n-1}e-\cdots-a_1A^2e-a_0Ae~({\rm mod}~p).
\end{equation}

Plugging Eq.~\eqref{Eq1} into Eq.~\eqref{Eq2} gives that

\begin{eqnarray}\label{Eq3}
\begin{split}
A^{n}e&\equiv&(-a_{n-3}+a_{n-2}^2)A^{n-2}e+(-a_{n-4}+a_{n-3}a_{n-2})A^{n-3}e+\cdots+\\
&&(-a_1+a_2a_{n-2})A^{2}e+(-a_0+a_1a_{n-2})Ae+a_0a_{n-2}e~({\rm mod}~p).
\end{split}
\end{eqnarray}

It follows from Lemma~\ref{rank} that $e,Ae,\ldots,A^{n-2}e$ are linearly independent over $\mathbb{F}_p$. Comparing Eq.~\eqref{CH1} and Eq.~\eqref{Eq3}, we get that
 \begin{eqnarray}\label{EQQ}
 a_0a_{n-2}&=&-c_n,\nonumber\\
 -a_0+a_1a_{n-2}&=&-c_{n-1},\nonumber\\
 -a_1+a_2a_{n-2}&=&-c_{n-2},\nonumber \\
 \vdots\\
 -a_{n-4}+a_{n-3}a_{n-2}&=&-c_{3},\nonumber\\
 -a_{n-3}+a_{n-2}^2&=&-c_{2}.\nonumber
 \end{eqnarray}

Iterating the above equations from the bottom up, it is easy to verify that $\phi(a_{n-2})\equiv0~({\rm mod}~p)$. Similarly, suppose $W(H)x\equiv~0~({\rm mod}~p)$ has a solution $(b_0,b_1,\ldots,b_{n-2},1)^T$,
 we must have $\phi(b_{n-2})\equiv0~({\rm mod}~p)$. Note that every $a_i$ (resp. $b_i$) is uniquely determined by $a_{n-2}$ (resp. $b_{n-2}$) for $i=0,1,\ldots,n-3$.
 By the assumption that $\phi(x)\equiv 0~({\rm mod}~p)$ has no two distinct roots, we get $a_{n-2}=b_{n-2}$, and hence $a_i=b_i$ for $i=0,1,\ldots,n-3$. The proof is complete.

\end{proof}

\begin{rem}
We give a straightforward explanation why $\phi(a_{n-2})=0$ holds, over $\mathbb{F}_p$. Let $C$ be the companion matrix of $\phi(x)$.
Let $\eta=(a_0,a_1,a_2,\ldots,a_{n-2},1)^T$. Then it is easy to verify that Eq.~\eqref{EQQ} is equivalent to the following equation:
\begin{eqnarray*}
\left(\begin{array}{cccccc}
 a_{n-2}&0&0&\cdots&0&c_n\\
 -1&a_{n-2}&0&\cdots&0&c_{n-1}\\
 0&-1&a_{n-2}&\cdots&0&c_{n-2}\\
 \vdots&\vdots&\vdots&\ddots&0&\vdots\\
 0&0&0&\cdots&a_{n-2}&c_{2}\\
 0&0&0&\cdots&-1&a_{n-2}
\end{array}
\right)
\left(\begin{array}{c}
 a_0\\
 a_1\\
 a_2\\
 \vdots\\
 a_{n-2}\\
 1
\end{array}
\right)=O,
\end{eqnarray*}
i.e., $(a_{n-2}I-C)\eta=O$. Therefore, $a_{n-2}$ is an eigenvalue of $C$, i.e., $\phi(a_{n-2})=0$.
\end{rem}

Unfortunately, the assumption in Lemma~\ref{oddroot} that $\phi(x)\equiv 0~({\rm mod}~p)$ has no two distinct roots does not always hold. We shall get rid of this difficulty
by considering the family of matrices $A+tJ$ and $B+tJ$ ($J$ is the all-one matrix) for all integers $t$ instead of $A$ and $B$.

Now, let $$W_t(G):=[e,A_te,\ldots,A_t^{n-2}e,A_t^{n-1}e]~{\rm and}~ W_t(H):=[e,B_te,\ldots,B_t^{n-2}e,B_t^{n-1}e],$$ where $A_t=A+t J$, $B_t=B+t J$, and $t\in \mathbb{Z}$. Define $\phi(x,t)=\det(x I-A_t(G))$. By Lemma~\ref{oddroot}, we can easily get the following corollary.

\begin{Cor} Suppose that $\phi(x,t_0)\equiv 0~({\rm mod}~p)$ has no two distinct roots for some integer $t_0$.
Suppose that ${\rm rank}_p(W(G))={\rm rank}_p(W(H))=n-1$. Then $W_{t_0}(G)x\equiv 0~({\rm mod}~p)$ and $W_{t_0}(H)x\equiv 0~({\rm mod}~p)$ have the same set of solutions.
 \end{Cor}
 \begin{proof}It follows from $Q^{\T}AQ=B$ and $Qe=e$ that $Q^{\T}A_tQ=B_t$. Then replace respectively $A$ and $B$ with $A_t$ and $B_t$ in the proof of Lemma~\ref{oddroot},
 it is easy to see the corollary holds.

 \end{proof}

 \begin{lemma}\label{lemodd}
 Under the conditions of Theorem~\ref{Main}, suppose that $p\mid\ell$, then there exists an integer $t_0$ such that $\phi(x,t_0)\equiv 0~({\rm mod}~p)$ has no two distinct roots.
 \end{lemma}

 \begin{proof} It is easy to see that $\phi(x,t)=(1+t)\phi(G,x)-(-1)^nt\phi(\bar{G},-1-x)$. Thus we may write
 $\phi(x,t)=\varphi(x)(\phi_1(x)+t\phi_2(x))$, where $\gcd(\phi_1(x),\phi_2(x))=1.$

 Now we show that $\varphi(x)$ always has a factor $x-\lambda_0$ for some $\lambda_0\in{\mathbb{F}_p}$.
 Actually, it follows from  $Q^{\T}AQ=B$ and $Qe=e$ that $Q^{\T}W(G)=W(H)$, and hence $W^{\T}(G)(\ell Q)=\ell W^{\T}(H)\equiv~0~({\rm mod}~p).$
 Note that ${\rm rank}_pW(G)=n-1$. It follows that ${\rm rank}_p \ell Q=1$. Let $z\not\equiv~0~({\rm mod}~p) $ be any column of $\ell Q$. Then it follows from
 $AQ=QB$ that $Az\equiv~\lambda_0z~({\rm mod}~p) $ for some integer $\lambda_0$. Further note that $e^{\T}z\equiv~0~({\rm mod}~p)$, we have $A_tz=(A+tJ)z\equiv~\lambda_0z~({\rm mod}~p)$ for any $t\in\mathbb{Z}$. Therefore, we have $(x-\lambda_0)\mid \phi(x,t)$, over $\mathbb{F}_p$.
 Thus, $(x-\lambda_0)\mid \varphi(x)$, since $\gcd(\phi_1(x),\phi_2(x))=1$.

 Now we claim, suppose that $\phi_1(x)+t\phi_2(x)\equiv~0~({\rm mod}~p)$ has a root $x_t\in {\mathbb{F}_p}$ for every $t\in{\{0,1,\ldots,p-1\}}$,
 we must have $x_i\neq x_j$ for $i\neq j$. For otherwise, suppose that $\phi_1(x)+t_1\phi_2(x)\equiv~0~({\rm mod}~p)$ and $\phi_1(x)+t_2\phi_2(x)\equiv~0~({\rm mod}~p)$
 have a common root $\hat{x}$, for $t_1\neq t_2$. It is easy to see $(x-\hat{x})\mid \phi_1(x)$ and $(x-\hat{x})\mid \phi_2(x)$; a contradiction.

 Therefore, there must exist a $t_0\in{\mathbb{F}_p}$ such that $\phi_1(x)+t_0\phi_2(x)\equiv~0~({\rm mod}~p)$ has a root $\lambda_0$ by the Pigeonhole Principle.
 Note that $\phi_1(x)+t_0\phi_2(x)\equiv~0~({\rm mod}~p)$ has no other root than $\lambda_0$ (otherwise it contradicts the above claim).

 It remains to show that $\varphi(x)$ has no root other than $\lambda_0$. This will be proved in the next lemma.
 Combining the above facts, whenever there is a $t\in{\mathbb{F}_p}$, for which $\phi_1(x)+t\phi_2(x)\equiv~0~({\rm mod}~p)$ has no root, we are done!
 When $\phi_1(x)+t\phi_2(x)\equiv~0~({\rm mod}~p)$ has a root for every $t\in{\mathbb{F}_p}$, still the lemma is true. This completes the proof.

 \end{proof}

 \begin{lemma}\label{lemodd1} Under the condition of Lemma~\ref{lemodd},
 $\varphi(x)$ has no roots other than $\lambda_0$.
 \end{lemma}

 \begin{proof}Suppose $\varphi(\lambda_1)\equiv~0~({\rm mod}~p)$ with $\lambda_1\neq \lambda_0$. Then for every $t=0,1,\ldots,p-1$, there exists a $\beta_t$ such that $(A+tJ)\beta_t=\lambda_1\beta_t$. It follows that $(\lambda_1I-A)\beta_t=t(e^{\T}\beta_t)e$. Clearly $e^{\T}\beta_t\neq 0$, for otherwise, we have $W_t^{\T}(G)\beta_t=0$.
 As $e^{\T}z\equiv~0~({\rm mod}~p)$, we get that $\beta_t$ and $z$ are linearly independent. Thus we get a contradiction since ${\rm rank}_p(W_t)=n-1$.

 For $t=0$, we have $(\lambda_1I-A)\beta_0=0$. For $t=1$, we have
  \begin{equation}\label{EE}
 (\lambda_1I-A)\beta_1=(e^{\T}\beta_1)e.
 \end{equation}
 Left-multiplying both sides of Eq.~\eqref{EE} by $\beta_0^{\T}$ we get
 $$0=\beta_0^{\T}(\lambda_1I-A)\beta_1=(e^{\T}\beta_1)(e^{\T}\beta_0),$$
 which contradicts the fact that $e^{\T}\beta_0\neq 0$ and $e^{\T}\beta_1\neq0.$

 This completes the proof.

 \end{proof}

 Now, we are ready to present the proof of Theorem~\ref{Main}~(i).

 \begin{proof}[Proof of Theorem~\ref{Main}~(i)]
 If $p\nmid \ell$, then it is obviously that $\ell\mid d_n(W(G))/p$ according to Lemma~\ref{rem1}. So we assume $p\mid \ell$ henceforth.

 First suppose that ${\rm rank}_pW(H)<n-1$. Let $d_n(W(G))=p^{t}b$ and $p\nmid b$, $d_n(W(H))=p^{t'}b'$ and $p\nmid b'$.
 Then $t'<t$. By Lemma~\ref{rem1}, we have $\ell\mid\gcd(p^{t}b,p^{t'}b')$. Note that $\gcd(p^{t}b,p^{t'}b')\mid p^{t'}b$. Thus, we have $\ell \mid p^{t-1}b=d_n(W(G))/p$.

  Next, it needs only to consider the case that ${\rm rank}_pW(H)=n-1$, and hence ${\rm rank}_pW(G)={\rm rank}_pW(H)=n-1$.

 It follows from $Q^{\T}AQ=B$ and $Q^{\T}e=e$ that $Q^{\T}A_tQ=B_t$. Thus, we have $Q^{\T}A_t^ke=B_t^ke$ for $k=0,1,\ldots,n-1$.
By Lemma~\ref{lemodd}, there exists a $t_0\in{\mathbb{F}_p}$ such that $\phi(x,t_0)$ has no two distinct roots over $\mathbb{F}_p$.
 Thus, there exists an integral vector $\eta:=(a_0,a_1,\ldots,a_{n-2},1)^T$ such that $$W_{t_0}(G)\eta\equiv 0~({\rm mod}~p)~{\rm and}~ W_{t_0}(H)\eta \equiv 0~({\rm mod}~p).$$ Let $$M_{t_0}(G)e:=a_0e+a_1A_{t_0}e+\cdots+a_{n-2}A_{t_0}^{n-2}e+A_{t_0}^{n-1}e$$ and $$M_{t_0}(H)e:=a_0e+a_1B_{t_0}e+\cdots+a_{n-2}B_{t_0}^{n-2}e+B_{t_0}^{n-1}e.$$
 Then we have $Q^{\T} M_{t_0}(G)e=M_{t_0}(H)e$.
 Define $$\hat{W}_{t_0}(G):=[e,A_{t_0}e,\ldots,A_{t_0}^{n-2}e,M_{t_0}(G)e/p]$$
 and $$\hat{W}_{t_0}(H):=[e,B_{t_0}e,\ldots,B_{t_0}^{n-2}e,M_{t_0}(H)e/p].$$
 Then $Q^{\T}\hat{W}_{t_0}(G)=\hat{W}_{t_0}(H)$. Note that $d_{n}(\hat{W}_{t_0}(G))=\frac{d_n(W(G))}p$ according to Lemma~\ref{lSNF}. We have $\ell\mid \frac{d_n(W(G))}p$ according to Lemma~\ref{rem1}.
 This completes the proof of Theorem~\ref{Main}~(i).

 \end{proof}

 \subsection{The case $p=2$}
In this subsection, we present the proof of Theorem~\ref{Main}~(ii). For the case $p=2$, we have $r={\rm rank}_2W(G)=\lceil\frac{n}2\rceil$. We shall show that Eq.~\eqref{key1} holds and the solution can be given explicitly using the coefficients of the characteristic polynomial of the graph $G$, as we shall see later.

We need the following lemmas.

\begin{lemma}[Wang~\cite{wang2017JCTB}]\label{leM}
Let $S$ be an integral symmetric matrix. If $S^2\equiv0\pmod2$, then $Se\equiv0\pmod2$.
\end{lemma}

\begin{lemma}[Wang~\cite{wang2017JCTB}]\label{sachs}
Let $\phi(x)=x^{n}+c_{1}x^{n-1}+\cdots+c_{n-1}x+c_{n}$ be the characteristic polynomial of graph $G$. Then $c_{i}$ is even when $i$ is odd.
\end{lemma}

\begin{lemma}[Wang~\cite{wang2017JCTB}]\label{evenM}
Let $\phi(x)=x^{n}+c_{1}x^{n-1}+\cdots+c_{n-1}x+c_{n}$ be the characteristic polynomial of the graph $G$. Let
\begin{equation*}
M=M(G):=\begin{cases}
&A^{\lceil\frac{n}{2}\rceil}+c_{2}A^{\lceil\frac{n-2}{2}\rceil}+\cdots+c_{n-2}A+c_{n}I, ~\mbox{if n is even};\\ &A^{\lceil\frac{n}{2}\rceil}+c_{2}A^{\lceil\frac{n-2}{2}\rceil}+\cdots+c_{n-3}A^2+c_{n-1}A, ~\mbox{if n is odd}.
\end{cases}
\end{equation*}
Then
 $Me\equiv0\pmod2$.
\end{lemma}
\begin{proof}In view of the importance of the lemma, we shall give a short proof here for completeness. We only prove the lemma for the case that $n$ is even, the case that $n$ is odd can be proved in the similar way.
By Cayley-Hamilton's Theorem, we get
\begin{equation}\label{eqM1}
  A^{n}+c_{1}A^{n-1}+\cdots+c_{n-1}A+c_{n}I=0.
\end{equation}
 By Lemma~\ref{sachs}, $c_{i}$ is even when $i$ is odd. Then we can rewrite Eq.~\eqref{eqM1} as
 \begin{equation}\label{eqM2}
  A^{n}+c_{2}A^{n-2}+\cdots+c_{n-2}A^{2}+c_{n}I\equiv0\pmod2.
\end{equation}
Note that $M=A^{\frac{n}{2}}+c_{2}A^{\frac{n-2}{2}}+\cdots+c_{n-2}A+c_{n}I$ when $n$ is even, hence
 \begin{eqnarray*}
   M^{2} &=& (A^{\frac{n}{2}}+c_{2}A^{\frac{n-2}{2}}+\cdots+c_{n-2}A+c_{n}I)^{2} \\
     &\equiv& A^{n}+c_{2}^{2}A^{n-2}+\cdots+c_{n-2}^{2}A^{2}+c_{n}^{2}I\\
    &\equiv&A^{n}+c_{2}A^{n-2}+\cdots+c_{n-2}A^{2}+c_{n}I \\
    &\equiv&0\pmod2.
 \end{eqnarray*}
 Note that $M$ is an  integral symmetric matrix, according to Lemma~\ref{leM}, $Me\equiv0\pmod2$. The proof is complete.
\end{proof}

As an immediate consequence of Lemma~\ref{evenM}, we have

\begin{Cor}\label{evencase} Let $G$ and $H$ be two generalized cospectral graphs. Then there exist integers $a_0, a_1,\ldots,a_{r-1}$ such that
Eq.~\eqref{key1} holds for $p=2$.
\end{Cor}

Now, we are ready to present the proof of Theorem~\ref{Main}~(ii).

\begin{proof}[Proof of Theorem~\ref{Main} (ii)]
Since $G$ and $H$ have the same generalized spectrum, by Lemma~\ref{lortho}, we have $Q^{\T}A(G)Q=A(H)$ and $Qe=e$. It follows that
\begin{gather*}
  Q^{\T}e=e, \\
  Q^{\T}Ae=Be,\\
  \vdots \\
   Q^{\T}A^{r-1}e=B^{r-1}e,\\
  Q^{\T}A^{r}e=B^{r}e,
\end{gather*}

It follows from Corollary~\ref{evencase} that Eq.~\eqref{key1} holds. Multiplying by $a_0, a_1,\ldots,a_{r-1}$ on the first to the $r$-th equations, respectively, and adding them to both sides of the $(r+1)$-th equations generates $Q^{\T}M(G)e=M(H)e$. Furthermore, it is easy to verify that
 \begin{equation*}
Q^{\T}A^{i}M(G)e=B^{i}M(H)e,~ \mbox{for} ~i=1,2,\ldots,\left\lfloor n/2\right\rfloor-1.
 \end{equation*}
 Thus, $Q^{\T}\bar{W}(G)=\bar{W}(H)$, where $$\bar{W}(G)=[e,Ae,\ldots,A^{\lceil\frac{n}{2}\rceil-1}e,M(G)e,AM(G)e,\ldots,A^{\lfloor\frac{n}{2}\rfloor-1}M(G)e]$$ and $$\bar{W}(H)=[e,Be,\ldots,B^{\lceil\frac{n}{2}\rceil-1}e,M(H)e,BM(H)e,\ldots,B^{\lfloor\frac{n}{2}\rfloor-1}M(H)e].$$
 By Lemma~\ref{evenM}, we have $M(G)e\equiv M(H)e\equiv0\pmod2$. Dividing the last $\lfloor\frac{n}{2}\rfloor$ columns of $\bar{W}(G)$ and $\bar{W}(H)$ by $2$ simultaneously, it follows that
\begin{equation}\label{eqeven1}
 Q^{\T}\hat{W}(G)=\hat{W}(H),
\end{equation}
where $$\hat{W}(G)=[e,Ae,\ldots,A^{\lceil\frac{n}{2}\rceil-1}e,\frac{M(G)e}{2},\frac{AM(G)e}{2},\ldots,\frac{A^{\lfloor\frac{n}{2}\rfloor-1}M(G)e}{2}]$$ and $$\hat{W}(H)=[e,Be,\ldots,B^{\lceil\frac{n}{2}\rceil-1}e,\frac{M(H)e}{2},\frac{BM(H)e}{2},\ldots,\frac{B^{\lfloor\frac{n}{2}\rfloor-1}M(H)e}{2}].$$

It follows from Lemma~\ref{lSNF} that $\hat{d}_{n}=\frac{d_n}2$. By Lemma~\ref{rem1}, we get that $\ell\mid \frac{d_{n}}{2}$. This completes the proof.

\end{proof}
 Finally, combining the proofs of Theorem~\ref{Main} (i) and Theorem~\ref{Main} (ii), Theorem~\ref{Main} follows.

 \section{An example}

 In this section, we shall give an example to illustrate the powerfulness of Theorem~\ref{Main}.
Let $G$ be a graph with adjacency matrix $A(G)$ given as follows:
$$A(G)={\tiny{
\left(
\begin{array}{cccccccccccc}
 0 & 1 & 0 & 1 & 0 & 1 & 0 & 0 & 0 & 0 & 0 & 0 \\
 1 & 0 & 1 & 1 & 1 & 1 & 1 & 1 & 1 & 0 & 1 & 1 \\
 0 & 1 & 0 & 0 & 1 & 0 & 0 & 1 & 0 & 1 & 0 & 0 \\
 1 & 1 & 0 & 0 & 0 & 0 & 1 & 0 & 1 & 1 & 0 & 0 \\
 0 & 1 & 1 & 0 & 0 & 0 & 1 & 0 & 1 & 1 & 1 & 1 \\
 1 & 1 & 0 & 0 & 0 & 0 & 0 & 0 & 0 & 1 & 0 & 0 \\
 0 & 1 & 0 & 1 & 1 & 0 & 0 & 0 & 1 & 0 & 0 & 0 \\
 0 & 1 & 1 & 0 & 0 & 0 & 0 & 0 & 0 & 1 & 0 & 1 \\
 0 & 1 & 0 & 1 & 1 & 0 & 1 & 0 & 0 & 0 & 1 & 0 \\
 0 & 0 & 1 & 1 & 1 & 1 & 0 & 1 & 0 & 0 & 0 & 0 \\
 0 & 1 & 0 & 0 & 1 & 0 & 0 & 0 & 1 & 0 & 0 & 1 \\
 0 & 1 & 0 & 0 & 1 & 0 & 0 & 1 & 0 & 0 & 1 & 0 \\
\end{array}
\right)
_{12\times 12}}}.$$

It is easy to compute using Mathematica 12.0 that the SNF of $W(G)$ is $${\rm diag}(\underbrace{1,1,1,1,1,1}_6,\underbrace{2,2,2,2,2,2\times 5^2\times 1145387}_6).$$
 Then for any $Q\in{\mathcal{Q}(G)}$ with level $\ell$, we have $\ell\mid 5$ according to Theorem~\ref{Main}. Therefore $\ell=1$ or $\ell=5$; the case $\ell=25$ is impossible by Theorem~\ref{Main}.

 Actually, we can find a regular rational orthogonal matrix
 $Q\in{\mathcal{Q}(G)}$ with level 5 which is given as follows (we refer the interested reader to \cite{WangWangYu} for the details):
$$Q={\tiny{\frac{1}{5}
\left(
\begin{array}{cccccccccccc}
 2 & 2 & -1 & -1 & 1 & 1 & 3 & -2 & 0 & 0 & 0 & 0 \\
 0 & 0 & 0 & 0 & 0 & 0 & 0 & 0 & 5 & 0 & 0 & 0 \\
 2 & 2 & -1 & -1 & 1 & 1 & -2 & 3 & 0 & 0 & 0 & 0 \\
 3 & -2 & 1 & 1 & -1 & -1 & 2 & 2 & 0 & 0 & 0 & 0 \\
 0 & 0 & 0 & 0 & 0 & 0 & 0 & 0 & 0 & 5 & 0 & 0 \\
 1 & 1 & 2 & 2 & 3 & -2 & -1 & -1 & 0 & 0 & 0 & 0 \\
 -1 & -1 & 3 & -2 & 2 & 2 & 1 & 1 & 0 & 0 & 0 & 0 \\
 -1 & -1 & -2 & 3 & 2 & 2 & 1 & 1 & 0 & 0 & 0 & 0 \\
 0 & 0 & 0 & 0 & 0 & 0 & 0 & 0 & 0 & 0 & 5 & 0 \\
 0 & 0 & 0 & 0 & 0 & 0 & 0 & 0 & 0 & 0 & 0 & 5 \\
 -2 & 3 & 1 & 1 & -1 & -1 & 2 & 2 & 0 & 0 & 0 & 0 \\
 1 & 1 & 2 & 2 & -2 & 3 & -1 & -1 & 0 & 0 & 0 & 0 \\
\end{array}
\right)}}.$$
Then, it is easy to verify that $A'=Q^{\T}AQ$ is a (0,1)-matrix given as follows:
$$A'=Q^{\T}AQ={\tiny{\left(
\begin{array}{cccccccccccc}
 0 & 0 & 0 & 0 & 1 & 0 & 1 & 0 & 1 & 0 & 0 & 1 \\
 0 & 0 & 0 & 1 & 0 & 0 & 0 & 0 & 1 & 1 & 0 & 0 \\
 0 & 0 & 0 & 0 & 0 & 0 & 1 & 0 & 1 & 1 & 1 & 0 \\
 0 & 1 & 0 & 0 & 0 & 1 & 0 & 0 & 1 & 0 & 0 & 1 \\
 1 & 0 & 0 & 0 & 0 & 0 & 0 & 0 & 1 & 0 & 0 & 1 \\
 0 & 0 & 0 & 1 & 0 & 0 & 0 & 1 & 1 & 1 & 0 & 0 \\
 1 & 0 & 1 & 0 & 0 & 0 & 0 & 0 & 1 & 0 & 1 & 0 \\
 0 & 0 & 0 & 0 & 0 & 1 & 0 & 0 & 1 & 1 & 1 & 1 \\
 1 & 1 & 1 & 1 & 1 & 1 & 1 & 1 & 0 & 1 & 1 & 0 \\
 0 & 1 & 1 & 0 & 0 & 1 & 0 & 1 & 1 & 0 & 1 & 1 \\
 0 & 0 & 1 & 0 & 0 & 0 & 1 & 1 & 1 & 1 & 0 & 0 \\
 1 & 0 & 0 & 1 & 1 & 0 & 0 & 1 & 0 & 1 & 0 & 0 \\
\end{array}
\right)_{12\times 12}}}.$$
Thus $A'$ is the adjacency matrix of a graph $G'$ which is generalized cospectral with $G$ but non-isomorphic to $G$. It can be further proved that, up to isomorphism, the graph $G'$ is the \emph{only} graph that are generalized cospectral with $G$ but non-isomorphic to $G$; see~\cite{WangWangYu}.

\section{Conclusions}

In this paper, we have presented a new method which results in a stronger version of a theorem of Wang~\cite{wang2017JCTB}.
The new proof is quite straightforward and is much easier to follow than the original ones in~\cite{wang2013EJC,wang2017JCTB},
which gives a new framework in dealing with the problem of generalized spectral characterizations of graphs. As a future work, we
would like to explore the method in more general situations, e.g., for $p=2$, we require that $\rank_2 W=\lceil\frac{n}2\rceil$ holds
in Theorem~\ref{Main}~(ii). However, we suspect that this condition can be removed. Also, Theorem~\ref{Main}~(i) requires that $\rank_pW =n-1$
for odd prime $p$. It would be interesting to consider the general case $\rank_pW =r$ for $r\geq 1$.

\end{document}